\documentclass[12pt]{article}

\usepackage{amsthm}
\usepackage{amsopn}
\usepackage{amsfonts}
\usepackage{amsmath}
\usepackage{amssymb}
\usepackage{amsfonts}
\usepackage{amscd}

\newtheorem{theorem}{Theorem}[section]
\newtheorem{corollary}[theorem]{Corollary}

\newtheorem{lemma}[theorem]{Lemma}

\theoremstyle{definition}
 \newtheorem{remark}{Remark}[section]
\newtheorem{definition}{Definition}[section]
\newtheorem{notation}{Notation}[section]
\newtheorem{example}{Example}[section]
\newtheorem{assumption}{Assumption}[section]
\numberwithin{equation}{section}

\def\IN{\mathbb{N}}

\def\cF{{\cal F}}
\def\cG{{\cal G}}

\def\pr{\prime}

\def\cY{{\cal Y}}

\def\wh{\widehat}

\def\q{\quad}

\def\noi{\noindent}

\def\pr{\prime}

\def\la{\langle}
\def\ra{\rangle}
\def\wt{\widetilde}
\def\wh{\widehat}
\def\ol{\overline}
\def\ul{\underline}
\def\os{\overset}
\def\us{\underset}

\def\vs{\vskip.3cm}
\def\vsk{\vskip.5cm}

\begin{document}

\title{SEMIMARTINGALE DYNAMICS AND\\ 
 ESTIMATION FOR A SEMI-MARKOV CHAIN}

\vsk
\author{Robert J. Elliott\\
\\
University of South Australia, Adelaide, Australia\\
and\\
University of Calgary, Calgary, AB, Canada}

\date{April 23, 2019}

\maketitle
\begin{abstract}

We consider a finite state discrete time process $X.$  Without loss
of generality the finite state space can be identified with the set
of unit vectors $\{e_1,e_2,\dots,e_N\}$ with $e_i =
(0,\dots,0,1,0,\dots,0)^\pr \in R^N.$  For a Markov chain the times
the process stays in any state are geometrically distributed.  This
condition is relaxed for a semi-Markov chain.  We first derive the
semimartingale dynamics for a semi-Markov chain.  We then consider
the situation where the chain is observed in noise.  We suggest how
to estimate the occupation times in the states and  derive
filters and smoothers for quantities associated with the chain.

\vsk
\noi
{\bf Acknowledgments.} The support of NSERC and the Australian
Research Council is appreciated.
\end{abstract}

\newpage
\section{Introduction}\label{Sec1} 
The topic of investigation is a finite
state, discrete time process
\newline $X=\{X_k,k=0,1,\dots\}.$  We suppose,
without loss of generality, that the state space of $X$ is the set
$S=\{e_1,e_2,\dots,e_N\},$ $e_i=(0,\dots,0,1,0,\dots,0)^\pr \in R^N,$
of unit vectors in $R^N.$  If $X$ is a Markov chain, the times the
chain spends in any state are geometrically distributed random
variables.  For a semi-Markov chain this condition is relaxed and
the occupation times can have more general distributions. 
Semi-Markov chains are related to renewal processes and have been
used in applications since their introduction over 60 years ago. 
Their more general occupation times often are a better fit for
empirical data.  

The two main contributions of this note are

\begin{enumerate}
\item the semimartingale dynamics for the semi-Markov chain
\item the consequent extension to semi-Markov chains of the
filtering, smoothing and estimation results for hidden semi-Markov
chains, once an integer estimate of the occupation time has been
obtained.
\end{enumerate}

Earlier references on semi-Markov processes include the books by
Koski \cite{K}, Barbu and Limnios \cite{BL}, and  van~der~Hoek and
Elliott
\cite{VE}. References on filtering include Yu \cite{Y},
Krishnamurthy, Moore and Chung \cite{KMC} and 
Elliott, Limnios and Swishchuk \cite{ELS}.

Unfortunately the references \cite{ELS} and \cite{VE} include a few typing errors.

\section{Dynamics}\label{Sec2}

All processes are defined on a probability space $(\Omega  ,\cF,P).$

Our process of interest is the discrete time, finite state,
time-homogeneous semi-Markov chain $X=\{X_k,\,0\le k\}.$  Without
loss of generality the state space of the process can be identified
with the set of unit vectors
$$
S=\{e_1,e_2,\dots,e_N\}
$$
$$
e_i = (0,\dots,0,1,0,\dots,0)^\pr \in R^N.
$$
A semi-Markov chain is similar to a Markov chain except that the
time it spends in any state is no longer geometrically distributed.

\vs\noi
\begin{notation}\label{Not2.1} With $X_0\in S,$
we suppose  either $X_0\,,$ or its distribution
$p_0=(p^1_0,p^2_0,\dots,p^N_0)^\pr\in R^N,$ is known.

Write $\tau  _1$ for the first jump time of $X,$ when $X_{\tau  _1}
\ne X_0$ and more generally $\tau  _n$ for the $n^{\text{th}}$ jump
time.

$\cF = \{\cF_k,\,k=0,1,2,\dots\},$ where $\cF_k = \sigma 
\{X_0,X_1,\dots,X_k\},$ is the filtration generated by $X.$
\end{notation}

\vs\noi
\setcounter{definition}{1}
\begin{definition}\label{Def2.2}  $X$ is a semi-Markov chain if
$$
\begin{aligned}
P(X_{\tau  _{n+1}} &= e_j,\,\tau  _{n+1} -\tau  _n = m\vert 
\cF_{\tau  _n})\\
&= P(X_{\tau  _{n+1}} = e_j,\, \tau  _{n+1}-\tau  _n = m\vert 
X_{\tau  _n}=e_i)\\
&= P(\tau  _{n+1}-\tau  _n = m\vert  X_{\tau  _{n+1}} =
e_j,\,X_{\tau  _n} =e_i)P(X_{\tau  _{n+1}} = e_j\vert  X_{\tau  _n} = e_i)\\
&= f_{ji}(m) p_{ji}, \; \text{say},
\end{aligned}
$$
where 
$$
P(\tau  _{n+1}-\tau  _n = m\vert  X_{\tau  _{n+1}} = e_j,\,
X_{\tau  _n} = e_i) = f_{ij}(m)
$$
and 
$$
P(X_{\tau  _{n+1}} = e_j\vert  X_{\tau  _n} = e_i) = p_{ji}.
$$
\end{definition}

\vs\noi
\setcounter{assumption}{2}
\begin{assumption}\label{Ass2.3}
We shall suppose $f_{ij} (m)$ does not depend
on $e_j$ and write
$$
\begin{aligned}
P(\tau  _{n+1}-\tau  _n = m\vert  X_{\tau  _{n+1}} = e_j,\, X_{\tau 
_n} = e_i) 
&= P(\tau  _{n+1} -\tau  _n = m\vert  X_{\tau  _n}- e_i)\\
&= \pi  _i(m), \q m=1,2,\dots\,.
\end{aligned}
$$
That is, for each $i,\, 1\le i\le N,$ $\{\pi_i(m),\,m=1,2,\dots\}$
is a probability distribution on the positive integers.
\end{assumption}

Also, recall we assumed $X$ is homogeneous in time so the above
probabilities are independent of $n.$

\vs\noi
\setcounter{notation}{3}
\begin{notation}\label{Not2.4}
Write
$$
\begin{aligned}
G_i(k):&= P\big(\tau  _{n+1}-\tau  _n\le k\vert  X_n = e_i\big)\\
&= \sum^k_{m=1}\,\pi_i(m)\\
\\
F_i(k) : &= P\big(\tau  _{n+1} -\tau  _n >k\vert  X_n = e_i\big)\\
&= 1- G_i(k).
\end{aligned}
$$
Then
$$
\begin{aligned}
P(\tau  _{n+1}&=\tau  _n +k+1\vert  X_{\tau  _n+k}
=X_{\tau  _n} = e_i) \\
&= P(\tau  _{n+1}=\tau  _n+k+1\vert \tau  _{n+1}>\tau 
_n+k,X_{\tau  _n}=e_i)\\
&= \frac{\pi_i(k+1)}{F_i(k)} = \Delta  ^i(k), \, \text{say}.
\end{aligned}
$$
\end{notation}

\vsk\noi
\setcounter{definition}{4}
\begin{definition}\label{Def2.5}
For $1\le i\le N$ define the process
$h^i_k(X_k)$ by the recursion
$$
\begin{aligned}
h^i_0(X_0) &=\la X_0,e_i\ra\\
h^i_k(X_k) &= \la X_k,e_i\ra + \la X_k,e_i\ra\, 
\la X_k,X_{k-1}\ra \,h^i_{k-1}(X_{k-1}).
\end{aligned}
$$
Then $h^i_k(X_k)$ measures the amount of time $X$ has spent in state
$e_i$ since it last jumped to $e_i$.

If 
$$
h_k(X_k) = \sum^N_{i=1}\, h^i_k(X_k)
$$
then $h_0(X_0) = 1$ and
$$ 
h_k(X_k) = 1+\la X_k,X_{k-1}\ra\, h_{k-1}(X_{k-1}).
$$
Then $h_k(X_k)$ measures the amount of time since the last jump.
\end{definition}

\setcounter{theorem}{5}
\begin{lemma}\label{Lem2.6}
Suppose $i\ne j,$ \, $1\le i, \, j\le N.$
Then
$$
P\big(X_{t+1} = e_j\vert  X_t=e_i,\, h_t(X_t)\big)= p_{ji}  \Delta 
^i\big(h^i_t(X_t)\big).
$$
\end{lemma}

\begin{proof}
We first make some observations from the given data.  As
$$
X_t = e_i, \, h_t(X_t) = h^i_t(X_t).
$$
Also
$$
\tau  _n = t-h_t(X_t) +1 \q \text{so} \q X_t=X_{\tau  _n} = e_i\,.
$$
Finally 
$$
\tau  _{n+1} = t+1.
$$
Consequently
$$
\begin{aligned}
P\big(X_{t+1} &=e_j \vert  X_t = e_i,\, h_t(X_t)\big)\\
&= P\big(X_{\tau  _{n+1}} = e_j\vert  X_{\tau  _n}= e_i\;\text{and}\;
\tau  _{n+1}= \tau  _n + h_t(X_t)\big)\\
&= p_{ji}\Delta^i  \big(h_t(X_t)\big).
\end{aligned}
$$
\end{proof}

\vs\noi
\setcounter{remark}{6}
\begin{remark}\label{Rem2.7}
We are assuming there is a jump from $e_i$ to a different $e_j\,, \,
i\ne j,$ at time $t+1.$

So, given there is a jump
$$
\sum^n_{j=1\atop j\ne i}\, p_{ji}=1.
$$
\end{remark}

\vs\noi
\setcounter{theorem}{7}
\begin{corollary}\label{Cor2.8}
Under the same hypotheses, if there is not a jump at $t+1$
$$
\begin{aligned}
P\big(X_{t+1} &= e_i\vert  X_t = e_i, h_t(X_t)\big)\\
&= 1-\Delta  ^i\big(h^i_t(X_t)\big)\\
&= 1 - \Delta  ^i h^i_t(X_t)\,\sum^n_{j=1\atop j\ne i} \, p_{ji}\,.
\end{aligned}
$$
\end{corollary}

\vs\noi
\setcounter{notation}{8}\label{Not2.8}
\begin{notation}
Write for $k=1,2,\dots,$ \, $A(k)$ for the matrix with entries
$$
\begin{aligned}
a_{ii}(k) &= 1-\Delta  ^i(k)\\
a_{ji}(k) &= p_{ji}\Delta  ^i(k).
\end{aligned}
$$
\end{notation} 

\vs\noi
\setcounter{example}{9}\label{Ex2.9}
\begin{example}
For $N=3$
$$
A(k) = \begin{pmatrix}
1-\Delta  ^1(k) &p_{12}\Delta  ^2(k) &p_{13}\Delta  ^3(k)\\
p_{21}\Delta  ^1(k) &1-\Delta  ^2(k) &p_{23}\Delta  ^3(k)\\
p_{31}\Delta  ^1(k) &p_{32}\Delta  ^2(k) &1-\Delta  ^3(k)
\end{pmatrix}\,.
$$
\end{example}

\vs\noi
\setcounter{notation}{10}\label{Not2.10}
\begin{notation}
Define the matrices:
$$
\Pi = (p_{ij} , \, 1\le i,\, j\le N)
$$
where $p_{ii} = -1$ and $p_{ji}$ is, as above, for $i\ne j.$
$$
D(k) =\; \text{diag}\big(\Delta  ^1(k),\Delta  ^2(k),\dots, \Delta 
^N(k)\big).
$$
Then 
$$
A(k) = I + \Pi D(k)
$$
where $I$ is the $N\times N$ identity matrix.
\end{notation}

For the case when $N=3$
$$
\Pi=\begin{pmatrix} -1 &p_{12} &p_{13}\\
p_{21} &-1 &p_{23}\\
p_{31} &p_{32} &-1\end{pmatrix} \q\q
D(k) = \begin{pmatrix} \Delta  ^1(k) &0 &0\\
0 &\Delta  ^2(k) &0\\
0 &0 &\Delta  ^3(k)\end{pmatrix}\,.
$$

A key result is the following representation of the semi-Markov
chain $X.$

\setcounter{theorem}{11}
\begin{theorem}\label{Thm2.11}
The semi-Markov chain $X$ has the following semi-martingale dynamics:
$$
X_{k+1} = A\big(h_k(X_k)\big) X_k + M_{k+1} \in R^N.
$$
Here $M_{k+1}$ is a martingale increment:
$$
E[M_{k+1}\vert  X_k,h_k(X_k)] = 0\in R^N.
$$
\end{theorem}

\begin{proof}
We need only remark that we can write $h_k(X_k)$ in
$A\big(h_k(X_k)\big)$ rather than the different occupation times
$h^1(X_k),\dots,h^N(X_k).$ 

This is because the components of $X_k$ are in effect indicator
functions.  If $X_k = e_i$ the product $A\big(h_k(X_k)\big)e_i$
selects the $i^{\text{th}}$ column of $A_k\big(h_k(X_k)\big)$ and
the $h(k)$ will be that for $h^i_k(X_k).$
\end{proof}

\section{Full Markov Dynamics}\label{Sec3}

We now consider an infinite state space for the semi-Markov chain
and give Markov dynamics.

\begin{remark} The complete state of our semi-Markov chain $X$ at a
time $k$ is given by the state of the chain $X_k$ and the number of
time steps $h_k$ the chain has been in that state.  Suppose
$$
X_k = e_i \q \text{and} \q h_k = r.
$$
Then at time $k+1$ either $X_{k+1} = e_i$ and $h_{k+1} = r+1$ or
$X_{k+1} = e_j, \q j\ne i,$ and $h_{k+1} =1.$

That is the state space of the `complete' process is $S\times \IN,$
where
\newline $\IN\in \{1,2,\dots\}$ is the natural numbers. Then with the
above notation
$$
(e_i, r) \to (e_i,r+1) \q \text{with probability} \q 1-\Delta  ^i(r)
$$
or
$$
(e_i,r) \to (e_j,1) \q \text{with probability} \q p_{ji}\Delta 
^i(r).
$$
\end{remark}

Suppose $N=3$ so $S= \{e_1,e_2,e_3\}$.  A state space for this
3-state semi-Markov chain can be identified with countably many
copies of $S,$ so, for example $(e_1,1)$ corresponds to the infinite
vector $(1,0,0,\vert  0,\dots)^\pr.$

The first three components of this infinite column vector
correspond to the states $(e_1,1), (e_2,1), (e_3,1).$  The second
three correspond to $(e_1,2),(e_2,2),(e_3,2)$ and so on.  If at time
$1$ the chain is in state
$$
(e_1,1) = (1,0,0\vert  0,0,0\vert  0,0,\dots)^\pr
$$
this can become either
$$
(e_1,2) = (0,0,0\vert  1,0,0\vert  0,0,\dots)^\pr
$$
with probability $\big(1-\Delta  ^1(1)\big),$ or
$$
(e_2,1) = (0,1,0\vert  0,0,0\vert  0,0,\dots)^\pr
$$
with probability $p_{21}\Delta  ^1(1),$ or
$$
(e_3,1) = (0,0,1\vert  0,0,0\vert  0,0,\dots)^\pr
$$
with probability $p_{31}\Delta  ^1(1).$ 

There is then an infinite matrix $C$ which describes these
transitions.

In this 3 state case write
$$
\Pi(i) = \begin{pmatrix}
0 &p_{12}\Delta  ^2(i) &p_{13}\Delta  ^3(i)\\
p_{21}\Delta  ^1(i) &0 &p_{23}\Delta  ^3(i)\\
p_{31}\Delta  ^1(i) &p_{32}\Delta  ^2(i) &0\end{pmatrix}
$$
and
$$
D(i) = \begin{pmatrix}
1-\Delta  ^1(i) &0 &0\\
0 &1-\Delta  ^2(i) &0\\
0 &0 &1-\Delta  ^3(i)\end{pmatrix}\,.
$$
With $0$ representing the $3\times 3$ zero matrix
$$
C= \begin{pmatrix}
\Pi(1) &\Pi(2) &\Pi(3) &\cdots\\
D(1) &0 &0 &\cdots\\
0 &D(2) &0 &\cdots\\
\cdot &\cdot &\cdot &\cdots\\
\cdot &\cdot &\cdot &\cdots
\end{pmatrix}\,.
$$

If we write the enlarged vectors as $\ol X_k$ then the
semi-martingale dynamics of the semi-Markov chain can be written as
$$
\ol X_{k+1} = C\ol X_k + \ol M_{k+1} \in \ol S
$$
where $\ol S$ is the set of all infinite unit vectors of the form
$\{e_{ij}\}$ where $e_{ij} = \delta  _{ij}\,,$ the Kronecker delta
function $1\le i,$ \, $j\le \infty.$

This gives
$$
E [ \ol X_{k+1}\vert  \ol X_k] = C\ol X_k
$$
and
$$
E[\ol X_{k+1}\vert  \ol X_0] = C^{k+1}\,\ol X_0\,.
$$
However, the approximate estimates defined in Section 5.3 below are
possibly of  more use in practise.

\section{Observation Dynamics}\label{Sec4}

The Viterbi and smoothing results of \cite{EM} are now adapted to
this situation.

We suppose the semi-Markov chain $X$ is not observed directly. 
Instead there is an observation sequence
$y=\{y_0,y_1,\dots,y_k,\dots\}$ where
$$
y_k = c(X_k) + d(X_k)w_k\,.
$$

$\{w_k,\, k=0,1,2,\}$ is a sequence of i.i.d. $N(0,1)$ random
variables.  $c(\cdot)$ and $d(\cdot)$ are known real valued
functions. Note that any real function $g(X_k)$ takes only the
finite number of values $g(e_1), g(e_2),\dots, g(e_N).$

Write $g_k = g(e_k)$ and $g= (g_1,g_2,\dots,g_N)^\pr \in R^N.$  Then
$$
g(X_k) = \la g,X_k\,\ra\,.
$$
Consequently there are vectors $c= (c_1,c_2,\dots,c_N),$ \,
$d=(d_1,d_2,\dots,d_N)$ such that
$$
c(X_k) = \la c,X_k\,\ra\q \text{and}\q d(X_k) = \la d,X_k\,\ra\,.
$$
We suppose $d_k>0$ for $k=1,\dots,N.$

\begin{remark}\label{Rem3.1}
We suppose the observation process $y$ is scalar valued.  The
extension to a vector $y$ is immediate.
\end{remark}

\newpage\noi
{\bf Reference Probability 4.2.}

We suppose  there is a second `reference' probability measure , $\ol
P,$ under which 

\begin{enumerate}
\item  the process $X$ is still a semi-Markov chain with dynamics
$$
X_{k+1}= A\big(h(k)\big)X_k+M_{k+1} \in R^N,
$$
\item  the process $y=\{y_0,y_1,\dots\}$ is a sequence of i.i.d.
$N(0,1)$ random variables. 
\end{enumerate}

From $\ol P$ we now construct the original probability $P$ under
which

\begin{enumerate}
\item  the process $X$ is a semi-Markov chain with dynamics
$$
X_{k+1} = A\big(h(k)\big)X_k + M_{k+1}
$$

\item  The process $w= (w_0,w_1,\dots)$ is a sequence of i.i.d.
$N(0,1)$ random variables where
$$
w_k = \frac{y_k -\la c,X_k\,\ra}{\la d,X_k\,\ra} \,.
$$
\end{enumerate}

\setcounter{definition}{2}
\begin{definition}\label{Def3.3}
For $k=0,1,2,\dots$ write
$$
\lambda  _k = \frac{\phi\big((y_k-\la c,X_k\,\ra)/\la
d,X_k\,\ra\,\big)}{\la d,X_k\,\ra\phi(y_k)}
$$
where $\phi(x)$ is the $N(0,1)$ density
$\frac{e^{-x^2}}{\sqrt{2\pi}}\,.$
$$
\Lambda  _{0,k} := \prod^k_{\ell=0}\,\lambda  _\ell\,.
$$
\end{definition}

Recall
$$
\cF = \sigma  \{X_0,X_1,\dots,X_k\}
$$
and write
$$
\begin{aligned}
\cY_k &= \sigma  \{y_0,y_1,\dots,y_k\}\\
\cG_k &= \sigma  \{X_0,\dots,X_k,y_0,\dots,y_k\}.
\end{aligned}
$$
We consider the related filtrations $\{\cF_k\},\,\{\cY_k\}$ and
$\{\cG_k\}.$

\setcounter{definition}{3}
\begin{definition}\label{Def3.4}
The original `real world' probability $P$ is defined in terms of
$\ol P$ by setting
$$
\frac{dP}{d\ol P}\,\big\vert  _{\cG_k} = \Lambda_{0  ,k}\,.
$$
\end{definition}

We can then prove

\setcounter{theorem}{4}
 \begin{lemma}\label{Lem3.5}
Under $P$ \, $X$ is a semi-Markov chain with $X_{k+1} =
A\big(h(X_k)\big)X_k + M_{k+1}$ and $\{w_k,\, k=0,1,\dots\}$ is a
sequence of {\rm i.i.d.}\,$N(0,1)$ random variables where
$$
w_k = (y_k -\la c,X_k\ra)/\la d,X_k\ra\,.
$$
\end{lemma}

That is, under $P$ \,$y_k =\la c,X_k\ra + \la d,X_k\ra\,w_k\,.$

\begin{proof} For a proof see \cite{EAM}.
\end{proof}

\section{Filter for $X$}\label{Sec5}

The basic problem now is to obtain a recursive estimate for
$$
E[X_k\vert  \cY_k\,].
$$
From Bayes' rule, see \cite{EAM}, this is
$$
\frac{\ol E[\Lambda_k X_k\vert\cY_k]}{\ol E[\Lambda  _k\vert  \cY_k]}
$$
where $\ol E$ denotes expectation with respect to $\ol P.$

Write $\ol q_k = \ol E[\Lambda  _kX_k\vert  \cY_k]\in R^N.$
$$
\begin{aligned}
\gamma  _j(y_{k+1}) &= \frac{\phi\big((y_{k+1}-\la c,e_j\,\ra)/\la
d,e_j\,\ra\big)}{\la d,e_j\,\ra\,\phi(y_{k+1})}\\
B(y_{k+1}) &= \begin{pmatrix}
\gamma  _1(y_{k+1}) &0\,\dots &0\\
0 &\gamma  _2(y_{k+1}) &0\\
\vdots &\ddots &\vdots\\
0 &\dots \,0 &\gamma  _k(y_{k+1})
\end{pmatrix} \,.
\end{aligned}
$$

\setcounter{notation}{0}
\begin{notation}\label{Not3.7}
Recall $h_k(X_k)$ is the number of time steps $X$ has stayed in state
$X_k$ since the last jump.
\end{notation}

We shall define below a $\cY_k$ measurable integer estimate $\wh
h_k$ of $h_k(X_k).$

This will be used in the recursion for $q.$

\setcounter{theorem}{1}
\begin{theorem}\label{Thm3.8}  With $\wh h_k$ a
$\cY_t\text{-measurable}$ integer estimate of $h_k(X_k)$ a recursive
estimate for $\ol q_k = \ol E[\Lambda  _kX_k\vert  \cY_k\,]$ is given
by
$$
q_{k+1} = B(y_{k+1})A(\wh h_k) q_k\,; \; q_0 = p_0\,.
$$
\end{theorem}

\begin{proof}
\begin{equation}\label{Eq3.1}
\begin{aligned}
\ol q_{k+1} &= \ol E[\Lambda    _{k+1}X_{k+1}\vert  \cY_{k+1}\,]\\
&= \ol E\big[\Lambda    _k\lambda  _{k+1}\big(A\big(h_k(X_k)\big)X_k +
M_{k+1}\big)\big\vert  \cY_{k+1}\,\big]\\
&=\ol E\big[\Lambda  _k\lambda 
_{k+1}A\big(h_k\big(X_k)\big)X_k\big\vert  \cY_{k+1}\,\big].
\end{aligned}
\end{equation}

Suppose $\wh h_k$ is a $\cY_k\text{-measurable}$ estimate of
$h_k(X_k)$ and substitute this in \eqref{Eq3.1}.

Then an estimate of \eqref{Eq3.1} is
\begin{equation}\label{Eq3.2}
\begin{aligned}
q_{k+1} :&= \ol E[\Lambda  _k\lambda  _{k+1}A(\wh h_k)X_k\vert 
\cY_{k+1}\,]\\
&= B(y_{k+1}A(\wh h_k)q_k
\end{aligned}
\end{equation}
where $q_k $ is the analogous recursive estimate of $\ol E[\Lambda  _k
X_k\vert  \cY_k\,].$

With $1= (1,1,\dots, 1)^\pr \in R^N$
$$
\la \,\ol q_{k+1},1\ra = \ol E[\Lambda  _{k+1}\vert  \cY_{k+1}\,]
$$
so a normalized estimate of $X_k$ given by $\cY_k$ is
$$
\frac{q_k}{\la q_{k}, 1\ra}\, .
$$
\end{proof}

\vs
\noi
{\bf Recursion For $\wh h$ 5.3}

Suppose $\wh h_k$ is known and $\cY_k$ measurable.  The MAP estimate
for $X_k$ given $\cY_k$ is the state $e_{i^*}$ corresponding to the
maximum value of $q^i_k,$ $i=1,2,\dots,N,$ or one of those states
in case of equality.  The recursion \eqref{Eq3.2} gives an estimate
for $q_{k+1}$ which in turn suggests a MAP estimate $e_{j^*}$ for
$X_{k+1}\,.$  That is $e_{j^*}$ corresponds to the maximum value of
$q^i_{k+1}\,,$ $i=1,2,\dots,N.$
Then if $i^*= j^*$\; $\wh h_{k+1} = \wh h_k +1.$  Otherwise $\wh
h_{k+1} = 1.$

These estimates for $h_k(X_k)$ provide corresponding estimates
$A(\wh h_k)$ for the transition matrix $A\big(h_k(X_k)\big).$    In
turn this provides the required recursive estimates.  

We noted in the representation of Theorem~2.12 that $h_k(X_k)$ could be
used rather than $h^i_k(X_k).$  In the recursion of Theorem~5.2 the
unnormalized probabilities $(q^1_k,q^2_k,\dots,q^N_k)$ are used. 
Observing these could give $\cY_k$ integer valued estimates $(\wh
h\,^i_k,\dots,\wh h_N)$ of $\big(h^1_k(X_k),\dots,h^N_k(X_k)\big).$ 
These can be used in the terms $a_{ji}(\wh h\,^i_k)$ of $A.$

\section{Estimates} \label{Sec6}
As stated earlier, the key contributions of this paper are:

\begin{enumerate}
\item   the semi-Martingale dynamics of Theorem \ref{Thm2.11}
for the semi-Markov chain.

\item   the recursive estimation $\wh h_k$ of the integer
$h_k(X_k)$ which is then substituted in the dynamics.
\end{enumerate}

This provides a recursive estimate of the state of the chain given
the observations $y.$

However, steps 1) and 2) above then allow immediate recursive
estimates for the following quantities:
$$
N^{ji}_k = \sum^k_{\ell=1}\,\la X_{\ell-1}, e_i\ra\,\la X_\ell,e_j\ra
$$
the number of jumps from $e_i$ to $e_j$
$$
J^i_k = \sum^k_{\ell=1}\,\la X_{\ell-1},e_i\ra
$$
the amount of time spent in state $e_i$ sums of the form
$$
G^i_k =\sum^k_{\ell=1}\, f(y_\ell) \,\la X_{\ell-1},e_i\ra\,.
$$

For example, consider the unnormalized estimate
$$
\ol \sigma  (N^{ji}_k X_k) := \ol E[\Lambda  _k N^{ji}_k X_k\vert 
\cY_k].
$$

A recursion for an approximation $\sigma  $ is given by

\begin{lemma}\label{Lem4.1} 
$$
\sigma  (N^{ji}_{k+1}X_{k+1}) = B(y_{k+1})A(\wh h_k)\sigma 
(N^{ji}_kX_k) + a_{ji}(\wh h_k)\,\la q_k,e_i\ra\,\gamma 
_j(y_{k+1})e_j
$$
where $q_k$ is given in Theorem~{\rm 5.2},
\end{lemma}

\begin{proof}
We suppose $q_k\,,$ $\sigma  (N^{ji}_kX_k)$ and $\wh h_k$ have been
determined.  Then
$$
\begin{aligned}
\ol E[\Lambda  _{k+1}&N^{ji}_{k+1}X_{k+1}\vert  \cY_{k+1}]\\
&\simeq \ol E[\Lambda    _k\lambda  _{k+1}(N^{ji}_k + \la
X_{k+1},e_j\ra\,\la X_k,e_i\ra)X_{k+1}\vert  \cY_{k+1}].
\end{aligned}
$$

Using Theorem 2.12 this is
$$
\ol E[\lambda  _{k+1}\Lambda  _k N^{ji}_k A(\wh h_k)X_k\vert  \cY_k]
+\ol E[\lambda  _{k+1}\Lambda  _k\,\la A(\wh h_k)X_k,e_j\ra \,
\la X_k,e_i\ra)\cY_k]e_j
$$
and the result follows.
\end{proof}

A recursion for an approximation $\sigma  $ of 
$$
\ol \sigma  (G^i_k) = \ol E[G^i_kX_k\vert  \cY_k]
$$
is given by

\setcounter{theorem}{1}
\begin{lemma}\label{Lem4.2}
$$
\sigma  (G^i_{k+1}X_{k+1})
= B(y_{k+1})A(\wh h_k)\sigma  (G^i_kX_k) + f(y_{k+1})\la
q_k,e_i\ra\,B(y_{k+1})A(\wh h_k)e_i\,.
$$
\end{lemma}

In applications we need $f(y) = 1, y$ or $y^2.$

For example, with $f(y)=1,$ $G_k^i = J^i_k$ and we have an
approximate recursion
$$
\sigma  (J^i_{k+1}X_{k+1}) 
= B(y_{k+1})A(\wh h_k)\sigma  (J^i_kX_k) +\la q_k,e_i\ra\,
B(y_{k+1})A(\wh h_k)e_i\,.
$$

\setcounter{remark}{2}
\begin{remark}\label{Rem4.3}
Now $\la X_k, \ul 1\ra = 1$ for all $k$ so $\la \sigma  (N^{ji}_k
X_k),1\ra$ gives an unnormalized estimate of $\sigma  (N^{ji}_k).$
\end{remark}

In turn these provide estimates such as
$$
\wh a_{ji}(\wh h_k) =\frac{\sigma  (N^{ji}_k)}{\sigma 
(J^i_k)}\,,\q\text{for}\q i\ne j,
$$
and for the other parameters of the model as in \cite{EAM}.

\section{Smoothers}\label{Sec7}
Suppose $0\le k\le T$ and we know
$\{y_0,y_1,\dots,y_T\}.$  We wish to find $E[X_k\vert  \cY_{T}].$
Using Bayes' theorem again, see \cite{EAM},
$$
E[X_k\vert  \cY_T] = \frac{\ol E[\Lambda  _{0,T}X_k\vert  \cY_T]}{\ol
E[\Lambda  _{0,T}\vert  \cY_T]}\,.
$$

Write $\Lambda  _{k+1,T} =\os T{\us{\ell=k+1}\prod}\, \lambda  _\ell\,.$
Then $\Lambda  _{0,T} = \Lambda  _{0,k}\Lambda  _{k+1,T}$ and
$$
\ol E[\Lambda  _{0,T}X_k\vert  \cY_{T}] =\ol E[\Lambda  _{0,k}X_k\ol
E[\Lambda  _{k+1,T}\vert  \cY_{T^v}\cF_k]\cY_T].
$$
The process $\{X_k,h_k(X_k)\}$ is Markov as shown in Theorem \q.

Therefore
$$
\ol E[\Lambda  _{k+1,T}\vert  \cY_{T^v}\cF_k] = \ol E[\Lambda 
_{k+1,T}\vert  \cY_{T^n}\sigma  \{X_k,h_k(X_k)\}].
$$
Suppose we have a sequence of MAP $\cY_T\text{-measurable}$ integer
estimates $\wh h_k$ of the $h_k(X_k).$

Substituting these in the dynamics of Theorem \ref{Thm2.11} , we have
approximate dynamics for $X$ under which it becomes a Markov chain.

\setcounter{definition}{0}
\begin{definition}\label{Def4.4}
Write $v^i_{T,T}=1$ and
$$
v^i_{k,T}: = \ol E[\Lambda  _{k+1,T}\vert  \cY_{T^v} \{X_k =
e_i\}_v\{h_k(X_k)=\wh h_k\}].
$$

Set $v_{k,T} = (v^1_{k,T},v^2_{k,T},\dots, v^N_{k,T})^\pr \in R^N.$
\end{definition}

\setcounter{theorem}{1}
\begin{theorem}\label{Thm4.5}
The process $v$ satisfies the backward dynamics
\begin{equation}\label{Eq4.1}
\begin{aligned}
v_{k,T} &= A^\pr (\wh h_k)B(y_{k+1})v_{k+1,t}\\
v_{T,T} &= (1,1,\dots,1)^\pr \in R^N.
\end{aligned}
\end{equation}
\end{theorem}

\begin{proof}
$$
\begin{aligned}
\la v_{k,T},e_i\ra &= v^i_{k,T} = \ol E[\Lambda  _{k+2,T}\lambda 
_{k+1}\vert  \cY_{T^v}\{X_k=e_i\}_v \wh h_k]\\
&= \sum^N_{j=1}\,\ol E[\la X_{k+1},e_j\ra \Lambda  _{k+2,T}\vert 
\cY_{T^v}\{X_k=e_i\}_v\wh h_k]\gamma  _j(y_{k+1})\\
&= \sum^N_{j=1}\,\ol E\big[\la X_{k+1},e_j\ra 
\ol E[\Lambda 
_{k+2,T}\vert  \cY_{T^v}\{X_k=e_i\}_v\wh h_{k^v}\{X_{k+1}e_j\}v\wh
h_{k+1}]\\
&\q \times\vert  \cY_{T^v}\{X_k=e_j\}_v \wh h_k\big]\gamma  _j(y_{k+1})\\
&=\sum^N_{j=1}\ol E[\la X_{k+1},e_j\ra\,\la v_{k+1,T},e_j\ra \vert 
\cY_{T^v}\{X_k=e_i\}_v\wh h_k]\gamma  _j(y_{k+1})\\
&= \sum^N_{j=1}\,a_{ji}(\wh h_k)\la v_{k+1,T}, e_j \ra \gamma 
_j(y_{k+1}).
\end{aligned}
$$
This gives the result of \eqref{Eq4.1}.
\end{proof}

\setcounter{theorem}{3}
\begin{theorem}\label{Thm4.6}
An unnormalized smoothed estimate for $X_k$ given observations
$\{y_0,y_1,\dots,y_T\}$ is 
$$
\wt q  _{k,T} := \ol E[\Lambda  _{0,T}X_k\vert  \cY_T] = (\text{\rm diag}\,
\la q_k,e_i\ra)v_{k,T}.
$$
A normalized smoothed estimate is then
$$
\frac{\wt q_{k,T}}{\la \wt q_{k,T},1\ra}\,.
$$
\end{theorem}

\begin{proof}
$$
\begin{aligned}
\ol E[\Lambda  _{0,T}\,\la X_k,e_i\ra\vert  \cY_T]
&= \ol E\big[\Lambda  _{0,T}\la X_k,e_i\ra \ol E[\Lambda  _{k+1,T}\vert 
\cY_{T^v}\{X_k=e_i\}_v \wh h_k]\cY_{0,T}\big]\\
& = \la q_k,e_i\ra\,\la v_{k,T},e_i\ra.
\end{aligned}
$$
Therefore,
$$
\begin{aligned}
\ol E[\Lambda  _{0,T}X_k\vert  \cY_{0,T}]
&= \sum^N_{i=1}\,\la q_k,e_i\ra\,\la v_{k,T},e_i\ra e_i\\
&= (\text{\rm diag}\,\la q_k,e_i\ra)v_{k,T}\,.
\end{aligned}
$$
\end{proof}

\section{Filtering the Full Markov Dynamics}\label{Sec8}

We showed in Section 3 that the state space of our
discrete time, finite state semi-Markov chain $X$ could be
identified with $\ol S=S\times \IN$ where 
\newline $S=\{e_1,e_2,\dots,e_N\}$
and $\ol N=\{1,2,3,\dots\}.$

We observed that by identifying this state space $\ol S$ with
countably many copies of $S,$ considered as an infinite column
vector, $X$ becomes a Markov chain $\ol X$ with transition matrix
$$
C = \begin{pmatrix} \Pi(1) &\Pi(2) &\Pi(3) &\cdots\\
D(1) &0 &0 &\cdots\\
0 &D(2) &0 &\cdots\\
\cdot &\cdot &\cdot &\cdots\\
\cdot &\cdot &\cdot &\cdots\end{pmatrix}
$$
so
$$
\ol X_{k+1} = C\,\ol X_k +\ol M_{k+1}\in \ol S
$$
where
$$
E[\,\ol X_{k+1} \vert  \,\ol X_k] = C\,\ol X_k\,.
$$

Suppose there is a scalar observation process
$y=\{y_k\,,\,k=1,2,\dots\}$ such that
$$
y_k =\la \,\ol c,\ol X_k\ra + \la\,\ol d,X_k\,\ra\,w_k
$$
where $w=\{w_k\,,\, k=1,2,\dots\}$ is a sequence of i.i.d. $N(0,1)$
random variables.

Here $\ol c = (c_1,c_2\,,\dots)$ and $\ol d = (d_1\,,d_i\,,\dots).$
Results of \cite{EM} extend immediately to this countable state
Markov chain so that, for example, if a reference measure $\ol P$ is
introduced, as in Section \ref{Sec3}, then with
$$
\begin{aligned}
\ol q_k &= \ol E\,[\Lambda    _k \ol X_k\vert  Y_k] \in R^\infty  \\
\gamma  _j(y_{k+1}) &= \frac{\phi\big((y_{k+1}-\la
\,\ol c,e_j\,\ra)/\la\, \ol d,e_j\,\ra\big)}{\la \,\ol d,e_j\,\ra
\phi(y_{k+1})}\\
\text{and}\q \ol B(y_{k+1}) &= \,\text{diag}\,\gamma  _i(y_{k+1})\\
\ol q_{k+1} &=  \ol B(y_{k+1})C\,\ol q_k\,.
\end{aligned}
$$

Filtering and smoothing results for quantities such as $N^{ij},
\,J^i$ and $G^i$ are exactly as before, except vectors with
countably many states are involved.

\section{Conclusion}\label{Sec9}
A semi-Markov chain is a generalization of a Markov chain where the
time spent in any state no longer has a geometric distribution.  We
derive the semimartingale dynamics of a semi-Markov chain.  Using
integer estimates of the times spent in a state we extend the
filter, smoothing and estimation results obtained in earlier
work.  In addition, by considering a countably infinite state space
the semi-Markov chain is a Markov chain and all earlier results
extend to this situation

\end{document}